\documentclass[A4paper,reqno,oneside,11pt]{amsart}
\usepackage[foot]{amsaddr}

\usepackage{fullpage}
\usepackage{setspace}
\usepackage{amsmath, amssymb, amsfonts, amsthm, bbm, mathtools, mathrsfs}
\usepackage{tikz}
\usepackage{subcaption}
\usepackage{verbatim,enumitem,graphics}
\usepackage{xcolor}
\usepackage[backref=page]{hyperref}
\hypersetup{%
    colorlinks,
    linkcolor={red!50!black},
    citecolor={green!50!black},
    urlcolor={blue!50!black}
}
\usepackage{dsfont}
\usepackage{lineno}
\usepackage[normalem]{ulem}
 \usepackage[T1]{fontenc}

\allowdisplaybreaks

\theoremstyle{plain}
\newtheorem{theorem}{Theorem}[section]
\newtheorem*{theorem*}{Theorem}
\newtheorem{lemma}[theorem]{Lemma}

\newtheorem{corollary}[theorem]{Corollary}

\theoremstyle{definition}

\newtheorem{claim*}{Claim}
\newtheorem{remark}[theorem]{Remark}
\newtheorem*{remark*}{Remark}

\newcommand{\N}{\mathbb{N}}
\newcommand{\Z}{\mathbb{Z}}

\newcommand{\coldeg}{d_{\mathrm{col}}}

\def\a{\alpha}

\def\l{\ell}

\makeatletter
\def\moverlay{\mathpalette\mov@rlay}
\def\mov@rlay#1#2{\leavevmode\vtop{%
   \baselineskip\z@skip \lineskiplimit-\maxdimen
   \ialign{\hfil$\m@th#1##$\hfil\cr#2\crcr}}}
\newcommand{\charfusion}[3][\mathord]{
    #1{\ifx#1\mathop\vphantom{#2}\fi
        \mathpalette\mov@rlay{#2\cr#3}
      }
    \ifx#1\mathop\expandafter\displaylimits\fi}
\makeatother

\newcommand{\dcup}{\charfusion[\mathbin]{\cup}{\cdot}}

\begin{document}

\setstretch{1.27}

\title{Improved Ramsey bounds for generalized Schur equations}

\author{Rafael Miyazaki}
\author{Eion Mulrenin}
\author{Cosmin Pohoata}
\author{Michael Zheng}
\address{Department of Mathematics, Emory University, Atlanta, GA 30322, USA}
\email{\{rafael.kazuhiro.miyazaki|eion.mulrenin|cosmin.pohoata|xzhe226\}@emory.edu}

\begin{abstract}
    We show that for $m, r \in \mathbb{N}$ and $N > (2m+1)^r (r!)^{1/m}$, every $r$-coloring of the integers in the interval $[N]$ contains a monochromatic solution to the equation
    \[
    x_1 + \dots + \dots x_{m+1} = y_1 + \dots + y_m.
    \]
    This generalizes and improves recent results of Ko\'scuiszko. We also show that if $N \geq 2^{r}$, then every $r$-coloring of the integers in $[N]$ must always determine a monochromatic solution to the above equation for some $m \geq 1$. The latter estimate is optimal. 
\end{abstract}

\maketitle


\section{Introduction}

Schur's theorem~\cite{Schur} is an early landmark result in Ramsey theory which states that for $r \in \N$ and $N > e r!$, every partition of the set of the first $N$ natural numbers $[N] = A_1 \dcup A_2 \dcup \dots \dcup A_r$ contains a solution to the equation $x + y = z$ with $x$, $y$, and $z$ (which need not all be distinct) all lying in the same part $A_i$ for some $1 \leq i \leq r$.
Typically, such a partition is called an {\it $r$-coloring}, the parts $A_i$ are called {\it color classes}, and a set of integers all lying in one color classes is called {\it monochromatic}, and we will use this terminology in the sequel.
In this language, one may concisely express Schur's theorem as the statement that for $N > er!$, every $r$-coloring of $[N]$ contains a monochromatic solution to the equation $x+y=z$. 

In his Ph.D. thesis, supervised by Schur, Rado (see, e.g., \cite[\S 2.5]{Promel}) gave a complete characterization of systems of linear questions which are {\it partition regular}, i.e., for which there exists a monochromatic solution under every $r$-coloring of $[N]$ for $N$ sufficiently large.
While his characterization for systems is rather complicated, it is easy to state for single equations: namely, Rado's theorem asserts that an equation $\sum_{i=1}^n \a_i x_i = 0$ with $\a_1, \dots, \a_n \in \Z$ constants and $x_1, \dots, x_n \in [N]$ variables is partition regular if and only if there exists a nonempty subset $I \subseteq [n]$ of indices for which $\sum_{i \in I} \a_i = 0$.

Thus, given such an equation and a number of colors, one may ask for a quantitative estimate on how large one must take $N$ for Rado's theorem to hold.
Cwalina and Schoen~\cite{CS17} introduced the study of the quantitative aspects of Ramsey properties of more generalized Schur equations of the form
\begin{equation}
\label{eq: longer-schur}
    x_1 + \dots + x_{m+1} = y_1 + \dots + y_m,
\end{equation}
which trivially satisfy Rado's condition. Here, we follow their notation and definition. Given $m, r \in \N$, we let $S_m(r)$ denote the least $N$ for which every $r$-coloring of the set $[N]$ produces a monochromatic solution to the equation~\eqref{eq: longer-schur}.

In this language, it is not hard to see that $S_{i}(r) \leq S_j(r)$ for $i > j$. 
However, even the qualitative growth of $S_1(r)$ remains unknown: while Schur's original paper \cite{Schur} gives $(3^r+1)/ 2 < S_1(r) \leq \lfloor e r!\rfloor +1$, progress on improving these bounds has been slow, with the current state of the art being $(\sqrt[5]{380})^r = (3.28\dots)^r \ll S_1(r) \leq (e- 1/6) r!$; see \cite{AH, ACPPRT22+,  E94, FS00, Wan97, Whitehead73, XXC02}. Meanwhile, Cwalina and Schoen showed that $S_2(r) \leq r^{-c \log r/ \log \log r} r!$ for an absolute constant $c$. This was recently improved by Ko\'sciuszko, who showed~\cite[Theorem 4]{K25} that $S_2(r) \leq 3^r \sqrt{(r+1)!}$. 

Our first result gives a new bound on $S_m(r)$ for $m \geq 3$ and all~$r$.

\begin{theorem}
\label{thm: main-balanced}
Let $m, r \in \N$.
If there exists an $r$-coloring of $[N]$ with no monochromatic
solution to \eqref{eq: longer-schur}, then $N \leq (2m+1)^r (r!)^{1/m}$. 
Equivalently,
\[
 S_m(r) \leq (2m+1)^r (r!)^{1/m}+1.
\]
\end{theorem}

One may also of course consider more generally equations of the form
\begin{equation}
\label{eq: schur-a-b}
    x_1 + \dots + x_a = y_1 + \dots + y_b
\end{equation}
for any pair $a, b \in \N$, which still trivially satisfy Rado's condition.
To that end, let $S_{a,b}(r)$ denote the least integer $N$ such that every $r$-coloring of $[N]$ contains a monochromatic solution to~\eqref{eq: schur-a-b}.
As a corollary of Theorem~\ref{thm: main-balanced}, we obtain a bound on $S_{a,b}(r)$ for certain combinations of $a, b \in \N$ and all~$r$.

\begin{corollary}
\label{cor: general-balanced}
Let $a > b$ be positive integers, let $d\coloneqq a-b$, and write
\[
 b=dm+w,
\qquad 0\le w<d.
\]
Assume that $m \geq 1$, or equivalently that $a \leq 2b$.
If there exists an $r$-coloring of $[N]$ with no monochromatic solution to \eqref{eq: schur-a-b}, then $N \leq (2m+1)^r (r!)^{1/m}$.
Equivalently,
\[
 S_{a,b}(r) \leq (2m+1)^r (r!)^{1/m}+1.
\]
\end{corollary}

As an application, note that for, e.g., $(a,b) = (12, 9)$ Corollary~\ref{cor: general-balanced} yields $S_{12,9}(r)\le 7^r (r!)^{1/3}+1$, which is much smaller than the $O(\sqrt{(r- k)!})$, $k = \Theta(\log r / \log \log r)$, bound from~\cite[Theorem 5]{K25}.
On the other hand, for pairs $a > b \in \N$ where Corollary~\ref{cor: general-balanced} does not apply (i.e. if $a > 2b$), we note that the repeated-Schur
argument from~\cite{K25} can be used to show that we always have that $S_{a,b}(r) \leq e(rL)!+1$, where $L \coloneqq \max \bigl\{ \lceil \log_2 a\rceil,\, \lceil \log_2 b\rceil+1 \bigr\}$. We omit the details. 

Finally, we also consider the problem of determining the least $N$ for which every $r$-coloring of $[N]$ contains a monochromatic solution to \eqref{eq: longer-schur} for {\it some} $m\in \N$. For our second result, we determine this $N$ exactly.   
\begin{theorem}\label{thm:arbitrary_m}
    Let $r\in \N$ and $N = 2^r$. In any $r$-coloring of $[N]$, there exists $m \in \N$ for which there is a monochromatic solution to the equation
    \[x_1+\cdots +x_{m+1} = y_1+\cdots+y_{m}.\]
    Furthermore, $N = 2^r$ is the minimum value for which the above property holds. In particular, $S_m(r)\ge 2^r$ for any $m\in \N$.
\end{theorem}

The threshold $2^r$ is reminiscent of the elementary graph-theoretic fact that if the edges of $K_N$ are colored with $r$ colors and each color class is bipartite, then $N\le 2^r$ (see, e.g.,~\cite{EG75}). 
Theorem~\ref{thm:arbitrary_m}, however, is not a direct consequence of this fact. 
When converting a coloring of $[N]$ into a difference-coloring of $K_{N+1}$ by coloring $\{a,b\}$ according to the color of $|a-b|$ (as in the standard proof of Schur's theorem by using Ramsey's theorem), an odd monochromatic cycle yields an equality between two monochromatic sums, but the two sides need not have numbers of terms differing by exactly one. 
Thus, it gives a different additive statement.
The proof of Theorem~\ref{thm:arbitrary_m} uses a sharper affine obstruction: a color class $A$ avoids equation \eqref{eq: longer-schur} for all $m \in \N$
if and only if $A$ is contained in a nonzero residue class modulo $d$ for some integer $d \geq 2$. 
This allows us to leverage instead a remarkable theorem by Crittenden and Vanden Eynden~\cite{C-VE69,C-VE70}. 
We discuss the proof of Theorem \ref{thm:arbitrary_m} in Section \ref{sect: proof-arbitrary_m}, along with some further quantitative aspects.

The remainder of the paper is organized as follows. First, in order to prove Theorem \ref{thm: main-balanced}, we make use of the recent improvement by Axenovich, Cames von Batenburg, Janzer, Michel, and Rundström~\cite{ACJMR25} on the state of the art for multicolor Ramsey numbers of odd cycles. To give slightly better bounds, we prove a sharpening of their key lemma which in turn also improves upon their bounds in Section~\ref{sect: proof-main-lem}; see Lemma~\ref{lem: main} and Remark~\ref{rem: Ramsey-odd-cycle-improv} respectively. In Section \ref{sect: proof-main-balanced}, we then use Lemma~\ref{lem: main} to prove Theorem~\ref{thm: main-balanced} and Corollary~\ref{cor: general-balanced}. 

\section{A Lemma of Axenovich, Cames von Batenburg, Janzer, Michel, and
Rundstr\"om}
\label{sect: proof-main-lem}

Recall that for a positive integer $q$, a {\it $q$-local edge-coloring} of a graph $G$ is an edge-coloring with an arbitrary number of colors in which every vertex is incident to at most $q$ of these colors.
Given an edge-coloring of a graph $G$, a vertex $x \in V(G)$, a color $c$, and $i \in \N$, we will write $N^i_c(x)$ to denote the vertices at distance exactly $i$ from $x$ in the color-$c$ subgraph of $G$, and we will set $N^{\leq i}_c(x) \coloneqq \bigcup_{j=0}^i N^j_c(x)$.
Note that $N^0_c(x) = \{x\} \subseteq N^{\leq i}_c(x)$.

By an argument of Erd\H{o}s, Faudree, Rousseau, and Schelp \cite{EFRS}, if a graph $H$ contains no cycle of length $2 \ell + 1$, then the chromatic number of $H[N^i(v)]$ for any $v \in V(H)$ and $i \in [\ell]$ is at most $2\ell -1$. Hence, given a $q$ edge-coloring of $G = K_n$ without a monochromatic $C_{2 \ell +1}$, it follows that $G[N_c^{\leq \ell}(v)]$ has bounded chromatic number, say at most $\chi$, for any $v \in V(G)$ and any color $c$. In a clever weighting argument, Axenovich et al.~\cite{ACJMR25} were then able to bound $n$ in terms of $q$, $\l$, and $\chi$, thus getting better upper bounds on the multicolor Ramsey numbers of odd cycles.

The result below is a slight sharpening of the key weighted lemma of Axenovich et al.~\cite[Lemma~2.1]{ACJMR25}. 
We replace their factor of $q^{q/\l}$ by~$(q!)^{1/\l}$.

\begin{lemma}
\label{lem: main}
Let $q,\l, \in \N$, $\chi \geq 2$,
and consider a $q$-local edge-coloring of a complete graph $G=K_n$.
Assume that for every vertex $v \in V(G)$ and every color $c$, the subgraph of color $c$
induced by $N_c^{\leq \l}(v)$ has chromatic number at most $\chi$. 
Then
\[
 n \leq \chi^q (q!)^{1/\ell}.
\]
\end{lemma}

\begin{proof}
Fix a $q$-local edge-coloring of $G = K_n$.
Following Axenovich et al.~\cite{ACJMR25}, for each vertex $x \in V(G)$, let $\coldeg(x)$ be the number of colors incident to $x$, and define the weight of $x$ by
\[
 w(x)\coloneqq \frac{1}{\chi^{\coldeg(x)}\, (\coldeg(x)!)^{1/\ell}}.
\]
For a set $U \subseteq V(G)$ of vertices, define the weight of $U$ by $w(U)\coloneqq \sum_{x\in U} w(x)$.
Since $\coldeg(x) \leq q$ for every $x$, we have $w(x) \geq \chi^{-q}(q!)^{-1/\ell}$,
and so it suffices to show that $w(V(G)) \leq 1;$ indeed, we would then have $n \cdot \chi^{-q}(q!)^{-1/\ell}\le w(V(G)) \leq 1$, which rearranges to the desired inequality.

We prove $w(V(G)) \leq 1$ by induction on $n = |V(G)|$, where the case $n=1$ holds trivially.
Assume now that $n \geq 2$ and let $v \in V(G)$ be a vertex of minimum color-degree.
Write $d \coloneqq  \coldeg(v)$. The proof now splits into two cases.

\noindent\textit{Case 1: $d=1$.}
Let $c$ be the unique color incident to $v$. 
Then every edge from $v$ has color $c$, and so $N_c^{\leq 1}(v) = V(G)$. By hypothesis, the color-$c$ subgraph induced by $V(G)$ has chromatic number at most $\chi$.
Fix a proper coloring of $V(G)$ (with respect to the color-$c$ subgraph) with $\chi$ colors, and choose a color class $S \subseteq V(G)$ of maximum weight. Then
$S$ spans no edge of color $c$, which implies $|S| < n$ since $v \notin S$, and has $w(S) \geq \frac{w(V(G))}{\chi}$. Now delete all vertices outside $S$. 
Every surviving vertex $x \in S$ loses the color $c$, so each such vertex has its weight multiplied by a factor of at least
\[
 \chi\left(\frac{\coldeg(x)!}{(\coldeg(x)-1)!}\right)^{1/\ell}
 =\chi\,\coldeg(x)^{1/\ell}\ge \chi.
\]
Hence the total weight of the remaining graph is at least $\chi\,w(S)\ge w(V(G))$. By the inductive hypothesis, since $|S| < n$, the new graph has total weight at most $1$, and therefore so does $G$.

\noindent\textit{Case 2: $d \geq 2$.}
Since exactly $d$ colors are incident to $v$, there exists a color $c$ such that $w(N_c(v)) \geq w(V(G)\setminus\{v\})/d$. Therefore,
\[
 w(N_c^{\leq 1}(v)) = w(v)+w(N_c(v))
 \ge w(v)+\frac{w(V(G))-w(v)}{d}
 > \frac{w(V(G))}{d}.
\]
We claim that there exists some $i\in [\l]$ with
\[
 w(N_c^{i+1}(v)) \leq (d^{1/\ell}-1)\,w(N_c^{\leq i}(v)).
\]
Indeed, if this failed for every $i \in [\l]$, then for each such $i$ we would have
\[
 w(N_c^{\leq i+1}(v))
 = w(N_c^{\leq i}(v))+w(N_c^{i+1}(v))
 > d^{1/\ell}w(N_c^{\leq i}(v)).
\]
Iterating this inequality gives
\[
 w(N_c^{\leq \ell+1}(v))
 >(d^{1/\ell})^{\ell}w(N_c^{\leq 1}(v))
 = d\cdot w(N_c^{\le 1}(v))
 > w(V(G)),
\]
which is impossible.

Now fix such an index $i$ and set $S \coloneqq  N_c^{\le i}(v)$ and $T\coloneqq N_c^{i+1}(v)$. Then $w(T) \leq (d^{1/\ell}-1) w(S)$. By hypothesis, the color-$c$ subgraph induced by $S$ has chromatic number at most $\chi$. Fix a proper $\chi$-coloring of that subgraph with w.l.o.g.~at least two color classes and let $S' \subseteq S$ be a color class of
maximum weight. Then $S'$ spans no edge of color $c$, $|S'| <|S|$, and $w(S')\ge w(S)/\chi$.
Now delete all vertices of $T\cup (S\setminus S')$.

We claim that every vertex of $S'$ loses the color $c$ entirely. 
Indeed, let $x\in S'$ and let
$y$ be a color-$c$ neighbor of $x$ in the original graph. 
Since $x \in N_c^{\leq i}(v)$, any such
$y$ lies in $N_c^{\leq i+1}(v) = S \cup T$. 
After the deletion, all vertices of $T$ are gone, all
vertices of $S \setminus S'$ are gone, and the remaining vertices of $S'$ span no color-$c$ edge.
So color $c$ is no longer incident to $x$ in the new graph.

Hence, every surviving vertex $x \in S'$ has its weight multiplied by
\[
 \chi\left(\frac{\coldeg(x)!}{(\coldeg(x)-1)!}\right)^{1/\l}
 =\chi \cdot \coldeg(x)^{1/\l}
 \geq \chi \cdot d^{1/\l},
\]
where the inequality $\coldeg(x)\ge d$ holds by the minimality of $d$. 
Every surviving vertex outside $S$ keeps the
same color-degree or loses colors, so its weight does not decrease. 
Therefore the new total
weight $W_{\mathrm{new}}$ satisfies
\[
 W_{\mathrm{new}}
 \geq w(V(G)) - w(S) - w(T) + \chi \cdot d^{1/\l} w(S').
\]
Using the two inequalities above we obtain
\[
 W_{\mathrm{new}}
 \ge w(V(G)) - w(S) - (d^{1/\ell}-1)w(S) + \chi \cdot d^{1/\ell} \frac{w(S)}{\chi}
 = w(V(G)),
\]
and so our deletion of vertices did not decrease the total weight. 
By the inductive hypothesis,
$W_{\mathrm{new}} \leq 1$, and so $w(V(G))\leq 1$ as well.
\end{proof}

\begin{remark}
    \label{rem: Ramsey-odd-cycle-improv}
    This sharpened version of~\cite[Lemma 2.1]{ACJMR25}
    immediately implies a lower-order improvement to the multicolor Ramsey numbers of odd cycles, directly following Axenovich et al.~\cite{ACJMR25}.
    In particular, writing $r(C_{2\ell+1}; q)$ for the $q$-color Ramsey number of $C_{2\ell+1}$, Lemma~\ref{lem: main} implies the bound
    \[
    r(C_{2\ell +1}; q) \leq (4\ell -2)^q (q!)^{1/\ell} + 1,
    \]
    which improves over the bound in~\cite[Theorem 1.1]{ACJMR25} by a factor of roughly $e^{q/\ell}$.
\end{remark}


\section{Proofs of Theorem~\ref{thm: main-balanced} and Corollary~\ref{cor: general-balanced}}

\label{sect: proof-main-balanced}

In proving Theorem~\ref{thm: main-balanced} and Corollary~\ref{cor: general-balanced}, we will use the following padding lemma several times.

\begin{lemma}[{\cite[Lemma 5]{K25}}]
\label{lem: padding}
Suppose that a set $A\subseteq \N$ contains a solution to
\begin{equation}
\label{eq: padding-1}
    x_1+\dots+x_u=y_1+\dots+y_v.
\end{equation}
Then, for all  integers $t \geq 1$ and $w \geq 0$, the same set $A$ contains a solution to
\begin{equation}
    \label{eq: padding-2}x_1+\dots+x_{ut+w}=y_1+\dots+y_{vt+w}.
\end{equation}
\end{lemma}

\begin{proof}
Fix a solution $x_1',\dots,x_u',y_1',\dots,y_v'\in A $
to equation~\eqref{eq: padding-1}. Repeat each $x_i'$ exactly $t$ times on the left, repeat each $y_j'$ exactly
$t$ times on the right, and then add the same element, say $x_1'$, exactly $w$ times to both
sides. 
The resulting equality is a solution to equation~\eqref{eq: padding-2}.
\end{proof}

\noindent We now prove Theorem~\ref{thm: main-balanced} by applying Lemma~\ref{lem: main} to equation \eqref{eq: longer-schur}.

\begin{proof}[Proof of Theorem~\ref{thm: main-balanced}]
    Fix an $r$-coloring of $[N]$ with no monochromatic solution to equation \eqref{eq: longer-schur}. Let $A_1,\dots,A_r \subseteq [N]$ be the color classes. 
    We define an auxiliary coloring $\Delta$ of the edges of the complete graph on
    vertex set $[N]$ by setting $\Delta(uv) = j$ if $|u-v| \in A_j$.
    Note that, since this is an $r$-coloring of $E(K_{[N]})$, it is trivially an $r$-local edge-coloring.

    We will verify the hypothesis of Lemma~\ref{lem: main} with $q=r$, $\ell=m$, $\chi=2m+1$. 
    To that end, fix a vertex $a \in [N]$ and a color $j \in [r]$.
    For each integer $t \in \{-m,-m+1,\dots, m\}$, define $B_t(a)$ to be the set of all
    $x \in [N]$ for which there exist integers $p,q\ge 0$ and elements $u_1,\dots,u_p,v_1,\dots,v_q\in A_j$
    such that
    \[
        p+q\le m,\qquad p-q=t, \qquad  x=a+u_1+\cdots+u_p-v_1-\cdots-v_q.
    \]
    We call $t$ the {\it charge} of such a representation.

    \begin{claim*}
        The sets $B_t(a)$ cover $N_j^{\leq m}(a)$.
    \end{claim*}

    \begin{proof}
        If $x \in N_j^{\leq m}(a)$, then by definition there exists a $j$-colored path (under $\Delta$) from $a$ to $x$ of length at most $m$. 
        Traversing this path from $a$ to $x$, each edge contributes either $+u$ or $-u$ for
        some $u\in A_j$ according to whether the next vertex is larger or smaller than the previous one. 
        Summing over these contributions along the path yields exactly a representation of $x-a$ of the required form. Hence,
        $x \in B_t(a)$ for some $t\in\{-m,\dots,m\}$.
    \end{proof}

    \begin{claim*}
        Each set $B_t(a)$ is independent in the color-$j$ subgraph.
    \end{claim*}

    \begin{proof}
        Fix $t$, and suppose that $x,y \in B_t(a)$ have $\Delta(xy) = j$.
        We may assume that $x>y$ so that $x-y\in A_j$.
        As $x,y \in B_t(a)$, we may choose representations of charge $t$
        \begin{align*}
            x &= a+u_1+\cdots+u_{p_1}-v_1-\cdots-v_{q_1}, \\
            y &= a+u_1'+\cdots+u_{p_2}'-v_1'-\cdots-v_{q_2}'
        \end{align*}
        with each $u_i, u_i', v_i, v_i'$ lying in $A_j$, $p_1+q_1\le m, p_2+q_2\le m$, and $p_1-q_1=p_2-q_2=t$. 
        Setting $s\coloneqq p_1+q_2=q_1+p_2$, 
        note that $2s=(p_1+q_1)+(p_2+q_2)\le 2m,$
        and so $s\le m$.

        Using our representations of $x$ and $y$, we may write $x-y \in A_j$ as
        \[
            x-y=(u_1+\cdots+u_{p_1})+(v_1'+\cdots+v_{q_2}')-(v_1+\cdots+v_{q_1})-(u_1'+\cdots+u_{p_2}'),
        \]
        and because all of the $u_i, u_i', v_i, v_i'$ are in $A_j$ as well, we obtain a monochromatic solution in color $j$ to
        \[
            z_1+\cdots+z_{s+1}=w_1+\cdots+w_s.
        \]
        By Lemma~\ref{lem: padding}, this yields a monochromatic solution in color $j$ to equation \eqref{eq: longer-schur}
        which did not occur in the original coloring of $[N]$.
    \end{proof}

    Since the $2m+1$ sets $B_{-m}(a),\dots,B_m(a)$ cover $N_j^{\leq m}(a)$ and each is independent,
    we conclude that the color-$j$ subgraph induced by $N_j^{\le m}(a)$ has chromatic number at
    most $2m+1$.
    Thus, Lemma~\ref{lem: main} applies and gives
    \[
        N \leq (2m+1)^r (r!)^{1/m},
    \]
    as promised.
\end{proof}

To conclude the section, we now prove Corollary~\ref{cor: general-balanced}.

\begin{proof}[Proof of Corollary~\ref{cor: general-balanced}]
    If a color class contains a solution to equation~\eqref{eq: longer-schur}
    then by Lemma~\ref{lem: padding} with parameters $t=d$ and the same $w$ it also contains a solution to
    \[
        x_1+\cdots+x_{d(m+1)+w}=y_1+\cdots+y_{dm+w},
    \]
    which is exactly equation~\eqref{eq: schur-a-b}.
    Therefore any coloring avoiding the latter equation also avoids the former one, and
    Theorem~\ref{thm: main-balanced} applies.
\end{proof}

\section{Proof of Theorem~\ref{thm:arbitrary_m}}
\label{sect: proof-arbitrary_m}

In this section, we will consider colorings of $[N]$ with no monochromatic solutions to equation \eqref{eq: longer-schur} for \emph{any} $m \in \N$. For the lower bound, i.e. that there exists an $r$-coloring of the first $N = 2^{r} - 1$ positive integers without any monochromatic solutions, we observe that the number of variables of the left and right hand side have different parity. Hence, we may consider the $r$-coloring of $[N]$ defined by $n\longmapsto \nu_2(n)+1$, where $\nu_2(n)$ is the $2$-adic valuation of $n$ (i.e., $\nu_2(n)=0$ if $n$ is odd, and $\nu_2(n)=1+ \nu_2(n/2)$ if $n$ is even). If $x_1,\dots,x_{m+1},y_1,\dots,y_m$ are of the same color, then
    \[\nu_2(x_1) = \dots =\nu_2(x_{m+1}) = \nu_2(y_1)=\dots=\nu_2(y_m)\]
    and thus, as $x_i / 2^{\nu_2(x_i)}, y_i / 2^{\nu_2(y_i)}$ are odd,
    \[\nu_2(x_1+\dots+x_{m+1}) \neq \nu_2(y_1+\dots+y_m).\]
    In particular, they can't form a solution to equation \eqref{eq: longer-schur}.
    
    In order to prove the other direction of Theorem~\ref{thm:arbitrary_m}, we resort to the following result which was first conjectured by Erd\H{o}s~\cite{Erdos62} and later proven by
    Crittenden and Vanden Eynden~\cite{C-VE69,C-VE70}. A short proof of a generalization of this result was later found by Balister, Bollob\'as, Morris, Sahasrabudhe, and Tiba~\cite{BBMST}.

    \begin{theorem}[Crittenden and Vanden Eynden~\cite{C-VE69,C-VE70}]\label{thm:AP_covers}
        Let $\mathcal{A} = \{A_1,A_2,\dots,A_k\}$ be a collection of $k$ arithmetic progressions. If $\mathcal{A}$ covers all integer numbers from $1$ to $2^k$, then it covers $\Z$.
    \end{theorem}

\begin{proof}[Proof of Theorem~\ref{thm:arbitrary_m}]
    By the coloring provided above, it suffices to show that any coloring\\
    $\psi \colon~[2^r]~\longrightarrow~[r]$ admits a monochromatic solution to equation \eqref{eq: longer-schur} for some $m \in \N$. For the sake of contradiction, assume otherwise and let $\psi$ be a counterexample.
    \begin{claim*}\label{claim:colors_invariant_mod_d}
        For every $s\in [r]$, there exist integers $d_s>1$ and $0<r_s<d_s$ such that $\psi^{-1}(s) \subset d_s\Z+r_s$. 
    \end{claim*}
    In other  words, Claim \ref{claim:colors_invariant_mod_d} says that every color class $\psi^{-1}(s)$ must be fully contained in some nonzero residue class modulo $d_{s}$, for some integer $d_s > 1$.  
    \begin{proof}
        Fix $s\in [r]$ and let $A = \{a_1,a_2,\dots,a_t\} =  \psi^{-1}(s)$. If $|A|\le 1$, the result follows trivially, and so we may assume $|A|\ge 2$. 
        
        For a set of integers $S$, we shall write $\gcd(S)$ for the greatest common divisor of all the elements from $S$. Using this notation, let $d = \gcd(A)$ and let $d' = \gcd(A-A)/d$. It suffices to show that $d'>1$. Indeed, notice that if we set $d_s = d'd$ and $d' >1$, then we can let $r_s$ be the remainder of an arbitrary (fixed) element of $A$ in the division by $d_s$. Since $d_s$ divides all differences in $A$, we have that $A\subset d_s \Z + r_s$. Furthermore, since $d<d_s$, we also have that $r_s\ne 0$.

         So, let us suppose that $d' = 1$ and seek a contradiction. The idea is to consider the set $B$ of integers $b_i = a_i/d$ for every $i \in [t]$. Note that a solution in $B$ to the equation \[x_1+\dots+x_{m+1} = y_1 + \dots + y_m\]
        would give a monochromatic solution of the same equation in $\psi$ by scaling back the solution by a factor of $d$. Hence, such solutions cannot exist.

      On the other hand, since $1= d' = \gcd(B-B)$, B\'ezout's theorem implies that there exist integers 
      $(c_{ij})_{(i,j)\in[t]^2}$ 
      such that
      $1 = \sum_{(i,j)\in[t]^2}c_{ij}(b_i-b_j)$
      Now let 
        \[k_i = \sum_{j=1}^t (c_{ji} - c_{ij})b_j\]
        and notice that
        \begin{equation}\label{eq:orthogonal_vector}
            \sum_{i=1}^t k_ib_i = \sum_{(i,j) \in [t]^2}  (c_{ji} - c_{ij})b_jb_i= 0
        \end{equation}
        and
        \begin{equation}\label{eq:sum_k}
            \sum_{i=1}^t k_i = \sum_{(i,j)\in[t]^2} (c_{ji} - c_{ij})b_j = \sum_{(i,j)\in[t]^2} c_{ij}b_i - \sum_{(i,j)\in[t]^2} c_{ij}b_j = 1.
        \end{equation}
    Now let $K^+ = \{i\in [t]\colon k_i \ge 0\}$ and $K^- = \{i\in [t]\colon k_i <0\}$. Then we choose $m = -\sum_{i\in K^-} k_i$ and notice that \eqref{eq:sum_k} gives $\sum_{i\in K^+} k_i= m+1$.
    We can now construct a solution in $B$ to the equation $x_1+\dots+x_{m+1} = y_1 + \dots + y_m$ as follows. 
    For each $i\in K^+$, let $b_i$ be in $(x_i)_{i\in[m+1]}$ with multiplicity $k_i$ and for each $i \in K^-$, let $b_i$ be in $(y_i)_{i\in[m]}$ with multiplicity $-k_i$. 
    Then 
    \[x_1+\dots+x_{m+1} - (y_1 + \dots + y_m) = \sum_{i \in K^+} k_i b_i - \sum_{i \in K^-} (-k_i b_i) = \sum_{i=1}^t k_i b_i = 0,\]
    where the last equality follows from \eqref{eq:orthogonal_vector}. This contradiction shows the assumption $d' = 1$ is false.
    \end{proof}
    Finally, notice that Claim~\ref{claim:colors_invariant_mod_d} implies that the family of arithmetic progressions \[\mathcal{A} = \{d_s\Z +r_s\colon s\in[r]\}\]
    covers $[2^r]$. Theorem~\ref{thm:AP_covers} would then imply that $\mathcal{A}$ covers $\Z$, but $0\not\in d_s\Z +r_s$ for all $s\in[r]$, which is a contradiction.
\end{proof}

\section{Quantitative aspects}
    Building on the aforementioned fact that every $r$-edge-coloring of the complete graph on $2^r+1$ vertices contains a monochromatic odd cycle, in 1975 Erd\H{o}s and Graham~\cite{EG75} asked for the shortest length of such a cycle which could be guaranteed in every such coloring.
    This problem has recently seen some exciting progress, for which we refer the interested reader to the papers~\cite{ACJMR25, GH24, JY25}. 
    In this spirit, our proofs of Theorems~\ref{thm: main-balanced} and Theorem~\ref{thm:arbitrary_m} can both be made quantitative.
    
    First, to make Theorem~\ref{thm:arbitrary_m} quantitative, we claim that if $A\subseteq [M]$ is not contained in a nonzero residue class modulo
    any integer $d\ge 2$, then $A$ contains a solution to equation~\eqref{eq: longer-schur} 
    for some $1\le m\le M-1$. Applying this with $M=2^r$ shows that in every $r$-coloring
    of $[2^r]$, the monochromatic equation in Theorem~\ref{thm:arbitrary_m} may be chosen
    with $m\le 2^r-1$.
    
    We remark that this claim cannot be improved to $m \neq M-2$. Indeed, the set $A = \{M-1,M\}\subseteq [M]$ is not contained in a nonzero residue class modulo any integer $d\ge 2$, yet $m = M-1$ is the minimum $m$ that yields solutions in $A$ to equation~\eqref{eq: longer-schur}. 

    We now prove the claim. By the argument above, the assumption on $A$ implies that $A$ contains a solution to equation~\eqref{eq: longer-schur} for at least one value of $m$. We first choose such a solution with $m$ minimal, and write it as an equality of multisets
    \[
        \sum_{x\in L}x=\sum_{y\in R}y,
        \qquad |L|=m+1,\qquad |R|=m,
    \]
    where $L$ and $R$ are both supported on $A$. Adjoin one zero to the shorter side, and put $P\coloneqq L$, $Q\coloneqq R\cup\{0\}$. Then $P,Q$ are multisets in $\{0,1,\dots,M\}$ satisfying
    \[
        |P|=|Q|=m+1,
        \qquad
        \sum P=\sum Q.
    \]
    This balanced identity is primitive in the following sense: there are no nonempty proper
    submultisets $P'\subset P$ and $Q'\subset Q$ such that $|P'|=|Q'|$ and $\sum P'=\sum Q'$. Indeed, if such $P',Q'$ existed and $0\in Q'$, deleting this zero from $Q'$ would give a smaller
    solution to one of our equations, contradicting the minimality of $m$; and if instead $0\notin Q'$,
    then passing to the complementary balanced identity $P\setminus P'$, $Q\setminus Q'$ gives a smaller balanced identity still containing the zero on the $Q$-side; deleting that zero
    again gives a smaller solution.

    Now, say
    \[
        P=\{p_1\le\cdots\le p_k\},\qquad
        Q=\{q_1\le\cdots\le q_k\},
        \qquad k=m+1,
    \]
    and define $d_i\coloneqq p_i-q_i$. Since $\sum P=\sum Q$, we have $\sum_i d_i=0$. Moreover, the primitivity property above implies
    that $(d_1,\dots,d_k)$ is actually a minimal zero-sum sequence over $\Z$, since no nonempty proper
    subcollection has sum zero. Let
    \[
        a\coloneqq \max\{d_i:d_i>0\},\qquad
        b\coloneqq \max\{-d_i:d_i<0\}.
    \]
    By a theorem of Lambert \cite{Lambert}
    (see also~\cite{Sissokho2014}), the number of positive terms among the $d_i$ is at most
    $b$, and the number of negative terms is at most $a$. The proof of this theorem is as follows. One can greedily reorder the terms of the sequence $(d_1,\dots,d_k)$ into $(e_1,\dots,e_k)$ so that the partial sums $s_i =\sum_{j=1}^i e_j$ satisfy $s_{i} \ge s_{i-1}$ if and only if $s_{i-1}\le 0$ for all $i \in [k]$. Notice that doing so implies that $s_{i-1}\in (-b,a]$ for all $i\in [k]$. Since $(d_1,\dots,d_k)$ is a minimal zero-sum sequence, the partial sums $s_0,s_1,\dots,s_{k-1}$ are pairwise distinct. At most $b$ of these are non-positive and at most $a$ of these are positive. Thus among the $d_i$ at most $b$ are positive and at most $a$ are negative. 
    
    In particular, this implies $k \le a+b$, and so to finish proving the claim, it suffices to check that $a+b\le M$. Choose indices $i,j$ such that $d_i=a$ and $d_j=-b$. If $i<j$, then $p_i\le p_j$ and $q_i\le q_j$, so
    \[
        a+b
        =(p_i-q_i)+(q_j-p_j)
        =(p_i-p_j)+(q_j-q_i)
        \le q_j-q_i\le M.
    \]
    If $j<i$, then similarly
    \[
        a+b=(p_i-p_j)+(q_j-q_i)\le p_i-p_j\le M.
    \]
 This concludes the proof. 

 On the other hand, from the proof of Theorem~\ref{thm: main-balanced} above, one can extract an analogue of~\cite[Theorem 1.2]{ACJMR25} which gets a much better bound  on $m$ in the regime where the interval is (much) larger than $2^r$.
Indeed, for $r \geq 3$ and $b \geq 2 e \ln r$, set
\begin{equation*}
    m_0  \coloneqq \frac{\ln r}{\ln \frac{b}{2 \ln r} + \ln \ln \frac{b}{2 \ln r}}.
\end{equation*}
If we have an $r$-coloring of $[N]$ with no monochromatic solution to equation \eqref{eq: longer-schur} for any $m \leq m_0$, then the proof above implies that 
\begin{align*}
    N < (2m_0)^r r^{r/m_0} 
    = \left(2 m_0 e^{\ln (r) / m_0} \right)^r = \left(b \cdot \frac{\ln \frac{b}{2 \ln r}}{ \ln \frac{b}{2 \ln r} + \ln \ln \frac{b}{2 \ln r}} \right)^r
        \leq b^r.
\end{align*}
Thus, it follows that for every $r \geq 3$ and $b \geq 2 e \ln r$, every $r$-coloring of the interval $[b^r]$ has a monochromatic solution to~\eqref{eq: longer-schur} for some $m \leq m_0$.

\section*{Acknowledgments}
C.P.\ was supported by NSF grant DMS-2246659. 
We would like to thank Lola Vescovo for her Math 532 final presentation on the work of Ko\'scuiszko~\cite{K25}, which inspired the present paper. 

We would also like to acknowledge the role of AI in preparing and enriching this manuscript. For example, Theorem \ref{thm: main-balanced} was initially supposed to be an upper bound of the form $S_{m}(r) \leq (2m+1)^{r} r^{r/m}+1$, using \cite[Lemma 2.1]{ACJMR25} as a blackbox. The improved Lemma~\ref{lem: main} was entirely produced by ChatGPT, 
as the result of an interaction that was initially meant to only clarify the proof from \cite{ACJMR25}. Similarly, the authors initially had an argument that the monochromatic equation in Theorem~\ref{thm:arbitrary_m} may be chosen with $m\le 2^{r+1}$ (using quantitative versions of B\'ezout's theorem). The idea to use Lambert's theorem to get the improved estimate $m \leq 2^{r}-1$ was due to ChatGPT.

\end{document}